\documentclass[12pt]{amsart}
\usepackage{times}
\usepackage{amsfonts}
\usepackage{amssymb}
\usepackage{bbm}
\usepackage{times}
\usepackage{amssymb}
\usepackage{amscd}
\usepackage{graphicx}

\usepackage{amsmath}
\usepackage{amssymb}

\usepackage{amsmath}
\usepackage{amsfonts}
\usepackage{amscd}
\usepackage{latexsym}
\usepackage{amsthm}
\usepackage{amssymb}
\usepackage{amsmath}

\theoremstyle{plain}
\newtheorem{theorem}{Theorem}[section]
\newtheorem{corollary}[theorem]{Corollary}
\newtheorem{lemma}[theorem]{Lemma}

\newtheorem{definition}[theorem]{Definition}

\def\dfrac{\displaystyle\frac}

\begin{document}

\vskip 0.5cm

\title[asymptotically tracially  approximation of ${\rm C^*}$-algebras]{{\bf  some properties   for  asymptotically tracially  approximation of ${\rm C^*}$-algebras}}
\author{Qingzhai Fan and Yutong Wu}
\address{ Qingzhai Fan\\ Department of Mathematics\\ Shanghai Maritime University\\
Shanghai\\China
\\  201306 }
\email{qzfan@shmtu.edu.cn}

\address{Yutong  Wu\\ Department of Mathematics\\ Shanghai Maritime University\\
Shanghai\\China
\\  201306 }
\email{wyt8524562023@163.com}



\begin{abstract} Let $\Omega$ be a class of unital $\rm C^{*}$-algebras.  The class
 of ${\rm C^*}$-algebras which are asymptotical   tracially  in $\Omega$, denoted  by ${\rm AT}\Omega$.
  In  this paper, we will show that  the following  class of
 ${\rm C^*}$-algebras in the class $\Omega$ are inherited by simple unital ${\rm C^*}$-algebras in the
 class  $\rm AT\Omega$
  $(1)$ the class of real rank zero ${\rm C^*}$-algebras,
   $(2)$ the class of ${\rm C^*}$-algebras with the radius of comparison $n$, and $(3)$ the class of $\rm AT\Omega$.
As an application, let $\Omega$ be a class of unital  ${\rm C^*}$-algebras  which  have  generalized tracial rank at most one (or has tracial topological rank zero, or has tracial topological rank one). Let $A$ be a unital separable  simple ${\rm C^*}$-algebra such that  $A\in \rm AT\Omega$, then $A$ has generalized tracial rank at most one (or has tracial topological rank zero, or has tracial topological rank one).
 \end{abstract}
\maketitle

\section{Introduction}

 The Elliott program for the classification of amenable
 ${\rm C^*}$-algebras might be said to have begun with the ${\rm K}$-theoretical
 classification of AF algebras in \cite{E1}. Since then, many classes of
 ${\rm C^*}$-algebras have been classified by the Elliott invariant.
    A major next step was the classification of simple
 $\rm AH$ algebras without dimension growth (in the real rank zero case see \cite{E6}, and in the general case
 see \cite{GL}).
 A crucial intermediate step was Lin's axiomatization of  Elliott-Gong's decomposition theorem for  simple $\rm AH$ algebras of real rank zero (classified by Elliott-Gong in \cite{E6}) and Gong's decomposition theorem (\cite{G1}) for simple $\rm AH$ algebras (classified by Elliott-Gong-Li in \cite{GL}). Heavily inspired by Gong's work in \cite{G1} and \cite{E6}, Lin introduced the concepts of   $\rm TAF$ and $\rm TAI$ (\cite{L0} and \cite{L1}). Instead of assuming inductive limit structure, Lin started with a certain abstract (tracial) approximation property.  This led eventually to the classification of  simple separable amenable stably finite  ${\rm C^*}$-algebras with finite nuclear dimension in the UCT class (see \cite{G4}, \cite{G5}, \cite{EZ5}, \cite{TWW1}).

   In the classification of  simple separable nuclear  ${\rm C^*}$-algebras,  it is  necessary to invoke some regularity property of the  ${\rm C^*}$-algebras. There are three regularity properties of particular interest: $\mathcal{Z}$-stability;  finite nuclear dimension; and certain comparison property of positive elements.
  Winter and Toms have conjectured that these three  properties are equivalent for all separable simple nuclear  ${\rm C^*}$-algebras (see \cite{ET}, \cite{TT1} and \cite{WW3}).

In order to be easier to  verify a  ${\rm C^*}$-algebra being $\mathcal{Z}$-stable; as well as
Hirshberg and Oroviz introduced tracial $\mathcal{Z}$-stability  in \cite{HO}, they  showed that a unital simple separable nuclear
  ${\rm C^*}$-algebra $A$  is   $\mathcal{Z}$-stable  if and only if  $A$ is tracially $\mathcal{Z}$-stable  in \cite{HO}.

Inspired by the work of Elliott,  Gong,  Lin, and  Niu in \cite {EGLN2}, \cite{GL2}, \cite{GL3}, and  by  the work of  Hirshberg and Oroviz's tracial  $\mathcal{Z}$-stability,
 in order to search a tracial version of Toms-Winter conjecture, Fu and Lin introduced  asymptotically tracially approximation of   ${\rm C^*}$-algebras in \cite{FL}, and also in \cite{FL}, Fu and Lin  introduced  some  concepts of  tracial nuclear dimension.  

In \cite{FL}, Fu and Lin show that the class of  stably finite ${\rm C^*}$-algebras; quasidiagonal ${\rm C^*}$-algebras;  purely infinite simple ${\rm C^*}$-algebras; stable rank one simple ${\rm C^*}$-algebras;  and the properties almost unperforated  are preserved to the  simple unital ${\rm C^*}$-algebras which are  asymptotically tracially in the same class.

In \cite{FF}, Fan and Fang show that the class of  certain   comparison  properties
${\rm C^*}$-algebras are preserved to the simple unital  ${\rm C^*}$-algebras which are  asymptotically
tracially in the same class.

In this paper, we will show that the following  result:

Let $\Omega$ be a class of unital $\rm C^{*}$-algebras.   The class
 of ${\rm C^*}$-algebras which are asymptotical   tracially  in $\Omega$, denoted  by ${\rm AT}\Omega$ (see Definition \ref{def:2.6}). Then $A\in\rm AT\Omega$ for any infinite-dimensional
		simple unital $\rm C^{*}$-algebra $A\in\rm AT(\rm AT\Omega)$ (see Definition \ref{def:2.6}).

As applications, we can get the following three results:

Let $\Omega$ be a class of unital  ${\rm C^*}$-algebras  which have tracial topological rank zero. Let $A$ be  a unital separable  simple ${\rm C^*}$-algebra such that  $A$ is asymptotically tracially in $\Omega$, then $A$ has tracial topological rank zero (see Definition \ref{def:2.3}).

Let $\Omega$ be a class of unital  ${\rm C^*}$-algebras  which have tracial topological rank one.  Let $A$ be  a unital separable  simple ${\rm C^*}$-algebra such that  $A$ is asymptotically tracially in $\Omega$, then $A$ has tracial topological rank one (see Definition \ref{def:2.3}).

Let $\Omega$ be a class of unital  ${\rm C^*}$-algebras  which  have  generalized tracial rank at most one.  Let $A$ be a unital separable  simple ${\rm C^*}$-algebra  such that $A$ is asymptotically tracially in $\Omega$, then $A$ has generalized tracial rank at most one (see Definition \ref{def:2.4}).

We also show that the following two main results:

Let $\Omega$ be a class of unital  ${\rm C^*}$-algebras  which have real rank zero.  Let $A$ be  a unital separable  simple ${\rm C^*}$-algebra  such that $A$ is asymptotically tracially in $\Omega$, then $A$ has real rank zero.

Let $\Omega$ be a class of stably finite unital
${\rm C^*}$-algebras which have $(n, m)$ comparison (see Definition \ref{def:2.1}). Let $A$ be a unital separable simple ${\rm C^*}$-algebra. If  $A$ is asymptotically tracially in $\Omega$, then $A$  has   $(n, m)$ comparison (see Definition \ref{def:2.1}).

As an application, let $\Omega$ be a class of stably finite unital
${\rm C^*}$-algebras such that ${\rm rc}(B)=n$ (see Definition \ref{def:2.1}) for any $B\in \Omega$. Let $A$ be a unital separable simple ${\rm C^*}$-algebra. If  $A$ is asymptotically tracially in $\Omega$, then ${\rm rc}(A)\leq n$ (see Definition \ref{def:2.1}).

 \section{Definitions and preliminaries}

Let $A$ be a  ${\rm C^*}$-algebra. Given  two positive elements $a, b \in A$, we call that  $a$ is Cuntz subequivalent to $b$, and write $a\precsim b$,  if there exist
$(s_n)_{n=1}^\infty$ in $A$, such that
$$\lim_{n\to \infty}\|s_nbs_n^*-a\|=0.$$

We call  that $a$ and $b$ are Cuntz equivalent (written $a\sim b$),
 if $a\precsim b$ and $b\precsim a$. We write $\langle a\rangle$ for the equivalence class of $a$.
(Cuntz equivalent for positive elements of ${\rm C^*}$-algebra was first introduced by  Cuntz in  \cite{CJ}. More recently results see \cite{APT}, \cite{APRT} and \cite{CEI}).

Given  a ${\rm C^*}$-algebra $A$, we denote  ${\rm M}_{\infty}(A)_+=\bigcup_{n\in {\mathbb{N}}}{\rm M}_n(A)_+$, and for $a\in {\rm M}_n(A)_+$ and $b\in{\rm M}_m(A)_+$, denote $a\oplus b:={\rm diag}(a, b)\in {\rm M}_{n+m}(A)_+$.

 Given $a, b\in {\rm M}_{\infty}(A)_+$, then there exist integer $n, m$, such that  $a\in {\rm M}_n(A)_+$ and $b\in{\rm M}_m(A)_+$. We call  that $a$ is Cuntz subequivalent to $b$ and write $a\precsim b$ if $a\oplus 0_{\max({(m-n)},0)}\precsim b\oplus 0_{\max({(n-m)},0)}$ as elements in $ {\rm M}_{\max{(n,m)}}(A)_+$.

Given  a positive element  $a$ in $A$,  and $\varepsilon>0,$
 we denote by $(a-\varepsilon)_+$ the element in $A$  via the functional calculus to the
 function $f(t)={\max (0, t-\varepsilon)},$  $t\in \sigma(a)$.

 Let $A$ be a stably finite unital ${\rm C^*}$-algebra. Recall that a positive element $a\in A$ is called purely positive if $a$ is
not Cuntz-equivalent to a projection. This is equivalent to saying that $0$ is
an accumulation point of $\sigma(a)$ (Recall that $\sigma(a)$ denotes the spectrum of $a$).

\begin{definition}{\rm (\cite{TF}.)}\label{def:2.1} Let $A$ be a unital stably finite   ${\rm C^*}$-algebra. Define the radius of comparison for $A$ to be
 ${\rm rc}(A)=\inf \{m/n:a\lesssim b$ whenever $na\oplus m 1_A\lesssim nb\}$.
We say that ${\rm C^*}$-algebra $A$ has $(n, m)$  comparison  whenever $na\oplus m1_A\lesssim nb$ implies that $a\lesssim b$  for any $a, b\in {\rm M}_{\infty}(A)_+.$
\end{definition}
The following facts are  well known.
\begin{theorem}{\rm (\cite{PPT}, \cite{EFF}, \cite{HO}, \cite{P3}, \cite{RW}.)} \label{thm:2.2} Let $A$ be a ${\rm C^*}$-algebra.

 $(1)$ Let $a, b\in A_+$ and any  $\varepsilon>0$  be such that
$\|a-b\|<\varepsilon$.  Then  $(a-\varepsilon)_+\precsim b$.

$(2)$ Let $a, p$ be positive elements in ${\rm M}_{\infty}(A)$ with $p$ a projection. If $p\precsim a,$ and $p$ is not Cuntz equivalent to $a$. Then there is a non-zero positive element $b$ in $A$ such that  $bp=0$ and $b+p\sim a$.

 $(3)$ Let $a$ be a purely  positive element  of $A$.
      Let  $\delta>0$. Then there exist a non-zero positive element $d\in A$ such that $ad=0$ and  $(a-\delta)_++d\precsim a.$

$(4)$ Let $a, b\in A$ satisfy $0\leq a\leq b$. Let $\varepsilon\geq 0$.
Then $(a-\varepsilon)_+\precsim(b-\varepsilon)_+$.
\end{theorem}

 We denote by $\mathcal{I}^{(0)}$ the class of finite dimensional ${\rm C^*}$-algebra, and denote by $\mathcal{I}^{(k)}$ the class of all unital ${\rm C^*}$-algebras which are untial hereditary ${\rm C^*}$-subalgebras of  ${\rm C^*}$-algebras of the form $C(X)\otimes F$ were $X$ is  a $k$-dimensional finite ${\rm CW}$ complex and $F\in \mathcal{I}^{(0)}$.

\begin{definition} {\rm (\cite{L0}, \cite{L1}, \cite{L2}.)}\label{def:2.3} A  unital simple  ${\rm C^*}$-algebra $A$ is said to
     have tracial topological rank no more than $k$ if,  for any
 $\varepsilon>0$, any finite
subset $F\subseteq A$, and any  non-zero element $a\geq 0$, there
exist a non-zero projection $p\in A$,   a ${\rm C^*}$-subalgebra $B\in \mathcal{I}^{(k)}$  with
$1_B=p$ such that

$(1)$ $\|px-xp\|<\varepsilon$ for any $x\in F$,

$(1)$  $pxp\in_{\varepsilon} B$ for all $x\in  F$, and

$(2)$ $1-p\precsim a$, or equivalently, $1-p$ is equivalent to a projection in ${\rm Her}(a)$.

\end{definition}

\begin{definition} {\rm (\cite{G4}.)}\label{def:2.4} Let $A$ be a unital simple  ${\rm C^*}$-algebra.   We say $A$ has the generalized tracial rank at most one, if the following hold:
Let $\varepsilon>0$, let $a\in A_+\backslash\{0\}$ and let $F\subseteq A$ be a finite subset. There exists a nonzero projection $p\in A$ and  ${\rm C^*}$-subalgebra $B$ which is a subhomogeneous  ${\rm C^*}$-algebra with one dimensional spectrum, or $B$ is finite dimensional  ${\rm C^*}$-algebra with $1_B=p$ such that

$(1)$ $\|px-xp\|<\varepsilon$ for any $x\in F$,

$(1)$  $pxp\in_{\varepsilon} B$ for all $x\in  F$, and

$(2)$ $1-p\precsim a$.

\end{definition}

In this case, we write $gTR(A)\leq 1$.

\begin{definition} \label{def:2.5}Given two ${\rm C^*}$-algebras  $A$ and $B$, let   $\varphi:A\to B$
be a map, let $\mathcal{G}\subset A$, and let $\varepsilon>0$. The map $\varphi$ is called $\mathcal{G}$-$\varepsilon$-multiplicative, or called $\varepsilon$-multiplicative on $\mathcal{G}$, if
for any $a, b\in \mathcal{G}$, $\|\varphi(ab)-\varphi(a)\varphi(b)\|<\varepsilon$. If, in addition, for any
$a\in \mathcal{G}$, $|\|\varphi(a)\|-\|a\||<\varepsilon$, then we say $\varphi$ is a $\mathcal{G}$-$\varepsilon$-approximate embedding.
\end{definition}

Let $A$ and $B$ be two ${\rm C^*}$-algebras. We say a map $\phi:A\to B$ is a c.p.c map when  $\phi$ is a completely positive contraction linear map. We call a linear map $\psi:A\to B$ is an
order zero map which means  preserving  orthogonality, i.e.,
 $\psi(e)\psi(f)=0$ for all $e, f\in A$ with $ef=0$,

	Let  $\Omega$ be a class of unital ${\rm C^*}$-algebras.  Then  the class
 of ${\rm C^*}$-algebras which are asymptotical   tracially   in $\Omega$, denoted  by ${\rm AT}\Omega$ was introduced by Fu and Lin in \cite{FL}.

 The following  definition is  Definition 3.1 in \cite{FL}.

\begin{definition}{\rm (\cite{FL}.)}\label{def:2.6}  Let $\Omega$ be  a class of ${\rm C^*}$-algebra. We call that  a unital  ${\rm C^*}$-algebra $A$ is  asymptotically tracially
in $\Omega$,  if  for any finite
subset $\mathcal{F}\subseteq A,$ any
 $\varepsilon>0,$   and any  non-zero  positive element $a,$ there are
a ${\rm C^*}$-algebra $B\in \Omega$,  and  c.p.c maps  $\alpha:A\to B$ and  $\beta_n: B\to A$, and $\gamma_n:A\to A$ such that

$(1)$ $\|x-\gamma_n(x)-\beta_n\alpha(x)\|<\varepsilon$ for all $x\in \mathcal{F}$, and for all $n\in {\mathbb{N}}$,

$(2)$ $\alpha$ is an $\mathcal{F}$-$\varepsilon$ approximate embedding,

$(3)$ $\lim_{n\to \infty}\|\beta_n(xy)-\beta_n(x)\beta_n(y)\|=0$, and $\lim_{n\to \infty}\|\beta_n(x)\|=\|x\|$ for all $x, y\in B$, and

$(4)$ $\gamma_n(1_A)\precsim a$ for all $n\in \mathbb{N}$.
\end{definition}

\textbf{Note that:} Let  $\mathcal{I}^{(0)}$ be the class of finite dimensional ${\rm C^*}$-algebra and let $P_1$ be the class of ${\rm C^*}$-algebra of 1-dimensional $\rm NCCW$ complexes respectively. Since $\mathcal{I}^{(0)}$ and $P_1$  are semiprojective. One easily verifies that $A$ is asymptotical tracially in  $\mathcal{I}^{(0)}$ is equivalent to that $A$ has tracial topological rank zero,  and $A$ is asymptotical tracially in  $P_1$ is equivalent to that $A$ has generalized tracial rank one. Also  $A$ is asymptotical tracially in  $\mathcal{I}^{(1)}$ is equivalent to that $A$ has tracial topological rank one (see \cite{FL}).

The following theorem is Proposition 3.8 in \cite{FL}.
\begin{theorem}(\cite{FL}.)\label{thm:2.7}  Let $\Omega$ be a class of ${\rm C^*}$-algebras. Let $A$ be a simple unital
${\rm C^*}$-algebra which is asymptotically tracially in  $\Omega$. Then the following conditions holds: For any finite
subset $\mathcal{F}\subseteq A,$ any
 $\varepsilon>0,$  and any  non-zero positive element $a,$ there are
a ${\rm C^*}$-algebra $B$  in $\Omega$ and c.p.c maps  $\alpha:A\to B$, and  $\beta_n: B\to A$, and $\gamma_n:A\to A\cap\beta_n(B)^{\perp}$ such that

$(1)$ the map $\alpha$ is unital  completely positive   linear map,  $\beta_n(1_B)$ and $\gamma_n(1_A)$ are projections and  $\beta_n(1_B)+\gamma_n(1_A)=1_A$ for all $n\in \mathbb{N}$,

$(2)$ $\|x-\gamma_n(x)-\beta_n\alpha(x)\|<\varepsilon$ for all $x\in \mathcal{F}$, and for all $n\in {\mathbb{N}}$,

$(3)$ $\alpha$ is an $\mathcal{F}$-$\varepsilon$-approximate embedding,

$(4)$ $\lim_{n\to \infty}\|\beta_n(xy)-\beta_n(x)\beta_n(y)\|=0$, and $\lim_{n\to \infty}\|\beta_n(x)\|=\|x\|$ for all $x, y\in B$, and

$(5)$ $\gamma_n(1_A)\precsim a$ for all $n\in \mathbb{N}$.
\end{theorem}

\begin{lemma}{\rm (\cite{FL}.)}\label{lem:2.8} If the class $\Omega$ is closed under tensoring with matrix algebras and
 under passing to  unital
hereditary ${\rm C^*}$-subalgebras, then  the class which is asymptotically tracially in $\Omega$  is closed under
tensoring with matrix algebras  and under passing to unital hereditary ${\rm C^*}$-subalgebras.
\end{lemma}

\section{the main result}

	\begin{theorem}\label{thm:3.1}
		Let $\Omega$ be a class of unital $\rm C^{*}$-algebras. Then  $A$ is asymptotically tracially in $\Omega$  for any infinite-dimensional
		simple unital $\rm C^{*}$-algebra $A\in\rm AT(\rm AT\Omega)$.
	\end{theorem}
	\begin{proof}
		We need to show that for any finite subset $F\subseteq A$, for any $\varepsilon>0$ and for any nonzero positive
		element $a\in A$, there exist a $\rm C^{*}$-algebra $D\in\Omega$, completely positive  contractive maps
		$\alpha:A\rightarrow D$, $\beta_{n}:D\rightarrow A$, and $\gamma_{n}:A\rightarrow A$ $(n\in{ \mathbb{N}})$
		such that
		
		$(1)$ $\|x-\gamma_{n}(x)-\beta_{n}\alpha(x)\|<\varepsilon$ for all $x\in F$, and for all $n\in  \mathbb{N}$,
		
		$(2)$ $\alpha$ is a $(F, \varepsilon)$-approximate embedding,
		
		$(3)$ $\mathop{\lim}\limits_{n\rightarrow\infty}\|\beta_{n}(xy)-\beta_{n}(x)\beta_{n}(y)\|=0$ and
		$\mathop{\lim}\limits_{n\rightarrow\infty}\|\beta_{n}(x)\|=\|x\|$ for all $x, y\in D$, and
		
		$(4)$ $\gamma_{n}(1_{A})\precsim a$ for all $n\in  \mathbb{N}$.
		
		Since $A$ is an infinite-dimensional simple unital $\rm C^{*}$-algebra, there exist positive elements
		$a_{1}, a_{2}\in A$ of norm one such that $a_{1}a_{2}=0$, $a_{1}\sim a_{2}$ and $a_{1}+a_{2}\precsim a$.
		
		Given $\varepsilon'\in(0,\varepsilon)$, with $G=F\cup\{a_{1}, a_{2}, 1_{A}\}$. Since $A\in\rm AT(\rm AT\Omega)$, by Theorem \ref{thm:2.7},
		there exist a $\rm C^{*}$-algebra $B\in\rm AT\Omega$, completely positive contractive maps
		$\alpha':A\rightarrow B$, $\beta_{n}':B\rightarrow A$, and $\gamma_{n}':A\rightarrow A\cap\beta_{n}'^{\perp}(B)$ $(n\in \mathbb{ N})$
		such that

		$(1')$ the map $\alpha'$ is unital  completely positive   linear map,  $\beta_n'(1_B)$ and $\gamma_n'(1_A)$ are projections and  $\beta_n'(1_B)+\gamma_n'(1_A)=1_A$ for all $n\in \mathbb{N}$,

		$(2')$ $\|x-\gamma_{n}'(x)-\beta_{n}'\alpha'(x)\|<\varepsilon'$ for all $x\in G$, and for all $n\in \mathbb{N}$,
		
		$(3')$ $\alpha'$ is a $(G,~\varepsilon')$-approximate embedding,
		
		$(4')$ $\mathop{\lim}\limits_{n\rightarrow\infty}\|\beta_{n}'(xy)-\beta_{n}'(x)\beta_{n}'(y)\|=0$ and
		$\mathop{\lim}\limits_{n\rightarrow\infty}\|\beta_{n}'(x)\|=\|x\|$ for all $x, y\in B$, and
		
		$(5')$ $\gamma_{n}'(1_{A})\precsim a_{1}$ for all $n\in  \mathbb{N}$.
		
		Since $B\in \rm AT\Omega$, for $\varepsilon'$  and $H=\{\alpha'(x):x\in G\}$, by Theorem \ref{thm:2.7}, there exist a $\rm C^{*}$-algebra $D\in\Omega$, completely positive  contractive maps
		$\alpha'':B\rightarrow D$, $\beta_{n}'':D\rightarrow B$, and $\gamma_{n}'':B\rightarrow B\cap\beta_{n}''^{\perp}(D)$ $(n\in  \mathbb{N})$
		such that
		
		$(1'')$ $\|\alpha'(x)-\gamma_{n}''\alpha'(x)-\beta_{n}''\alpha''\alpha'(x)\|<\varepsilon'$	 for all $x\in G$, and for all $n\in { \mathbb{N}}$,
		
		$(2'')$ $\alpha''$ is a $(H, \varepsilon')$-approximate embedding,
		
		$(3'')$ $\mathop{\lim}\limits_{n\rightarrow\infty}\|\beta_{n}''(xy)-\beta_{n}''(x)\beta_{n}''(y)\|=0$ and
		$\mathop{\lim}\limits_{n\rightarrow\infty}\|\beta_{n}''(x)\|=\|x\|$ for all $x, y\in D$, and
		
		$(4'')$ $\gamma_{n}''(1_{B})\precsim \alpha'(a_{2})$ for all $n\in \mathbb{N}$.
		
		Define $\alpha:A\rightarrow D$ by $\alpha(x)=\alpha''\alpha'(x)$, $\beta_{n}:D\rightarrow A$ by
		$\beta_{n}(x)=\beta_{n}'\beta_{n}''(x)$, and $\gamma_{n}:A\rightarrow A(n\in  \mathbb{N})$ by
		$\gamma_{n}(x)=\gamma_{n}'(x)+\beta_{n}'\gamma_{n}''\alpha'(x)$. Since for any $x\in A$, $\gamma_{n}'(x)\beta_{n}'\gamma_{n}''\alpha'(x)=0$, we have $\alpha, \beta_{n},\gamma_{n}$ are completely positive contractive maps.
		
		For all $x\in G$, we have
		
		 $\|x-\gamma_{n}(x)-\beta_{n}\alpha(x)\|=\|x-\gamma_{n}'(x)-\beta_{n}'\gamma_{n}''\alpha'(x)-
		\beta_{n}'\beta_{n}''\alpha''\alpha'(x)\|$
		
		 $\leq\|x-\gamma_{n}'(x)-\beta_{n}'\alpha'(x)\|+\|\beta_{n}'\alpha'(x)-\beta_{n}'\gamma_{n}''\alpha'(x)-
		\beta_{n}'\beta_{n}''\alpha''\alpha'(x)\|$
		
		$\leq\|x-\gamma_{n}'(x)-\beta_{n}'\alpha'(x)\|+
		\|\alpha'(x)-\gamma_{n}''\alpha'(x)-\beta_{n}''\alpha''\alpha'(x)\|$
		
		$<\varepsilon'+\varepsilon'<\varepsilon$, this is $(1)$.
		
		 By $(3')$, $(2'')$, and $\alpha(x)=\alpha''\alpha'(x)$, we can get
		
		 $\|\alpha(xy)-\alpha(x)\alpha(y)\|$

$=\|\alpha''\alpha'(xy)-\alpha''\alpha'(x)\alpha''\alpha'(y)\|$

$\leq \|\alpha''\alpha'(xy)-\alpha''(\alpha'(x)\alpha'(y))\|+
\|\alpha''(\alpha'(x)\alpha'(y))-\alpha''\alpha'(x)\alpha''\alpha'(y)\|$

$< 2\varepsilon'<\varepsilon$ for any $x, y\in F$.

We also have
		
		 $|\|\alpha(x)\|-\|x\||=|\|\alpha''\alpha'(x)\|-\|x\||$
		
		 $\leq|\|\alpha''\alpha'(x)\|-\|\alpha'(x)\||+|\|\alpha'(x)\|-\|x\||$
		
		 $<2\varepsilon$ for any $x\in F$.
		
		 Then $\alpha$ is a $(F, \varepsilon)$-approximate embedding, then we can get $(2)$.
		
		 By $(4')$ and $(3'')$, for $\varepsilon'$, there exists $N_{1}>0$, if $n>N_{1}$, then $\|\beta_{n}'(xy)-\beta_{n}'(x)\beta_{n}'(y)\|<\varepsilon'$, and for $\varepsilon'$, there exists $N_{2}>0$, if $n>N_{2}$, then $\|\beta_{n}''(xy)-\beta_{n}''(x)\beta_{n}''(y)\|<\varepsilon'$.
		
		 Let $N=\max\{N_{1},~N_{2}\}$, $n>N$, we have
		
		 $\|\beta_{n}(xy)-\beta_{n}(x)\beta_{n}(y)\|=\|\beta_{n}'\beta_{n}''(xy)-\beta_{n}'\beta_{n}''(x)\beta_{n}'\beta_{n}''(y)\|$
		
		 $\leq\|\beta_{n}'\beta_{n}''(xy)-\beta_{n}'(\beta_{n}''(x)\beta_{n}''(y))\|+
		 \|\beta_{n}'(\beta_{n}''(x)\beta_{n}''(y))-\beta_{n}'\beta_{n}''(x)\beta_{n}'\beta_{n}''(y)\|$
		
		 $<2\varepsilon'<\varepsilon$.
		
		 Similarly, $|\|\beta_{n}(x)\|-\|x\||<\varepsilon$, then we can get $(3)$.
		
		 Since
$\gamma_{n}(1_{A})=\gamma_{n}'(1_{A})+\beta_{n}'\gamma_{n}''\alpha'(1_{A})$,
 and

 		 $\gamma_{n}(1_{A})\sim (\gamma_{n}'(1_{A})-2\varepsilon)_++\beta_{n}'\gamma_{n}''\alpha'(1_{A})$

 $\sim \gamma_{n}'(1_{A})+(\beta_{n}'\gamma_{n}''\alpha'(1_{A})-2\varepsilon)_{+}$

 $\precsim\gamma_{n}'(1_{A})+(\beta_{n}'\alpha'(a_{2})-\varepsilon)_{+}$

 $\precsim a_{1}+a_{2}\precsim a$, then one has
		 $\gamma_{n}(1_{A})\precsim a$ for all $n\in \mathbb{N}$, this is $(4)$.
		
		 Therefore, $A$ is asymptotically tracially in $\Omega$.

	\end{proof}

\begin{corollary}\label{cor:3.2} Let $\Omega$ be a class of unital simple  $\rm C^{*}$-algebras such that $B$ has tracial topological  rank zero for any $B\in \Omega$.  Let $A$ be a simple unital $\rm C^{*}$-algebra such that $A$ is  asymptotically tracially
in $\Omega$. Then $A$ has tracial topological rank  zero.
\end{corollary}
\begin{proof} Since $A$ has tracial topological rank  zero if and only if $A$  is  asymptotically tracially in $\mathcal{I}^{(0)}$, by Theorem \ref{thm:3.1}, we can get this result.

\end{proof}
\begin{corollary}\label{cor:3.3} Let $\Omega$ be a class of unital simple  $\rm C^{*}$-algebras such that $B$ has tracial topological  rank  one  for any $B\in \Omega$.  Let $A$ be a simple unital $\rm C^{*}$-algebra such that $A$ is  asymptotically tracially
in $\Omega$. Then $A$ has tracial  topological rank one.
\end{corollary}
\begin{proof}Since $A$ has tracial topological rank  one if and only if $A$  is  asymptotically tracially in $\mathcal{I}^{(1)}$, by Theorem \ref{thm:3.1}, we can get this result.

\end{proof}
\begin{corollary}\label{cor:3.4}
 Let $\Omega$ be a class of unital simple  $\rm C^{*}$-algebras such that $B$ has  generalized tracial rank at most one  for any $B\in \Omega$.  Let $A$ be a simple unital $\rm C^{*}$-algebra such that $A$ is  asymptotically tracially
in $\Omega$. Then $A$ has generalized tracial rank at most one.
\end{corollary}
\begin{proof}Since $A$ has generalized tracial rank at most one if and only if $A$  is  asymptotically tracially in $P_1$, by Theorem \ref{thm:3.1}, we can get this result.

\end{proof}

Recall that a unital $\rm C^{*}$-algebra is said to have real rank zero, and written $RR(A)=0$, if the set of invertible self-adjoint elements is dense in $A_{sa}$ (the self-adjoint part of $A$).

The proof of the following Theorem is similar with Theorem 3.6.11 of \cite{L2}.

	\begin{theorem} \label{thm:3.5}
		Let $\Omega$ be a class of unital $\rm C^{*}$-algebras which have real rank zero. Let $A$ be a simple unital
		$\rm C^{*}$-algebra such that $A$ is asymptotical tracially in $\Omega$.
		Then $RR(A)=0$.
	\end{theorem}
	\begin{proof}
	   We need to show that for any $\varepsilon>0$, for any $a\in A_{sa}$ (we may assume that $\|a\|\leq 1$), there exist the invertible element $b\in A_{sa}$, such that
	   $\|a-b\|<\varepsilon$.
	
	   Define a continuous function $f\in C([-1,1])$ with $0\leq f\leq 1$, and

$$f(t)=\left\{
  \begin{array}{ll}
   1 & if\ |t|\geq \varepsilon/8\\
 linear & if\ \varepsilon/8< |t|<\varepsilon/4\\
  0 & if\ |t|\geq \varepsilon/4.
  \end{array} \right.
 $$

If $0\notin {\rm sp}(a)$, then $a$ is invertible. So we assume that $0\in {\rm sp}(a)$,
	   then $f(a)\neq0$. Since $A$ has real rank zero, then there exists nonzero projection $p_1\in {\rm Her}(f(a))$.

	   Define a continuous function $g\in C([-1,1])$, such that $0\leq g(t)\leq 1$, with

$$g(t)=\left\{
  \begin{array}{ll}
   1 & if\ |t|\geq \varepsilon/2\\
 linear & if\ \varepsilon/4< |t|<\varepsilon/4\\
  0 & if\ |t|\leq \varepsilon/4.
  \end{array} \right.
 $$

 Then we have $\|g(a)a-a\|<\varepsilon/2$ and $g(a)f(a)=0$. Set $a'=g(a)a$, then $a'p_1=p_1a'=0$ and $\|a-a'\|<\varepsilon/2$.

	   Set $A_{1}=(1-p_{1})A(1-p_{1})$, by Lemma \ref{lem:2.8},  $A_{1}\in\rm AT\Omega$, for any finite subset
	   $a'\in A_1$, and $\varepsilon'>0$ with $\varepsilon'<\varepsilon$,  by Theorem \ref{thm:2.7}, there exist $\rm C^{*}$-algebra $B\in\Omega$, completely positive contractive  maps
	   $\alpha:A_1\rightarrow B$, $\beta_{n}:B\rightarrow A_1$, and $\gamma_{n}:A_1\rightarrow A_1\cap\beta_{n}^{\perp}(B)$ $(n\in  \mathbb{N})$
	   such that
	
	   $(1)$ the map $\alpha$ is unital completely positive linear map, $\beta_{n}(1_{B})$ and $\gamma_{n}(1_{A})$
	   are projections and $\beta_{n}(1_{B})+\gamma_{n}(1_{A})=1_{A_{1}	}$ for all $n\in \mathbb{N}$,
	
	   $(2)$ $\|a'-\gamma_{n}(a')-\beta_{n}\alpha(a')\|<\varepsilon$,  and for all $n\in \mathbb{N}$,
	
	   $(3)$ $\alpha$ is a $(\{a'\},~\varepsilon)$-approximate embedding,
	
	   $(4)$ $\mathop{\lim}\limits_{n\rightarrow\infty}\|\beta_{n}(xy)-\beta_{n}(x)\beta_{n}(y)\|=0$ and
	   $\mathop{\lim}\limits_{n\rightarrow\infty}\|\beta_{n}(x)\|=\|x\|$ for all $x, y\in B$, and
	
	   $(5)$ $\gamma_{n}(1_{A_{1}})\precsim p_{1}$ for all $n\in  \mathbb{N}$.

	   Since $\alpha(a')\in B$,  and $B$ has real rank zero, for $\alpha(a')$, exist an invertible  element $b_{1}\in B_{sa}$ such that
	   $\|\alpha(a')-b_{1}\|<\varepsilon$, then $\|\beta_{n}\alpha(a')-\beta_{n}(b_{1})\|<\varepsilon$.
	
	   By $(4)$, for $\varepsilon'$, there  exists  $N>0$, if $n>N$, then
	
	   $\|\beta_{n}(b_{1}b_{1}^{-1})-\beta_{n}(b_{1})\beta_{n}(b_{1}^{-1})\|=\|\beta_{n}(1_{B})-\beta_{n}(b_{1})\beta_{n}(b_{1}^{-1})\|
	   <\varepsilon'$, and
	
	   $\|\beta_{n}(b_{1}^{-1}b_{1})-\beta_{n}(b_{1}^{-1})\beta_{n}(b_{1})\|=\|\beta_{n}(1_{B})-\beta_{n}(b_{1}^{-1})\beta_{n}(b_{1})\|
	   <\varepsilon'$,
	
	   Then $\beta_{n}(b_{1})$ is an invertible element of $\beta_{n}(1_{B})A\beta_{n}(1_{B})$.
	
	   Since $\beta_{n}(1_{B})+\gamma_{n}(1_{A})=1_{A_{1}}$, by $(5)$, then $1_{A_{1}}-\beta_{n}(1_{B})\precsim p_{1}$, i.e.,
	   $1_{A}-p_{1}-\beta_{n}(1_{B})\precsim p_{1}$.
	
	   Let $v\in A$ such that $v^{*}v=1_{A}-p_{1}-\beta_{n}(1_{B})$, $vv^{*}\leq p_{1}$.
	
	   Set $d=\gamma_{n}(a')+\dfrac{\varepsilon}{16}v+\dfrac{\varepsilon}{16}v^{*}+\dfrac{\varepsilon}{4}(p_{1}-vv^{*})$.
	  
By the  matrix representation,
   $\gamma_{n}(a')+\dfrac{\varepsilon}{16}v+\dfrac{\varepsilon}{16}v^{*}$ is invertible in
    $(1_{A_{1}}-\beta_{n}(1_{B})+vv^{*})A(1_{A_{1}}-\beta_{n}(1_{B})+vv^{*})$. Therefore, $d$ is an invertible self-adjoint element
    in $(1_{A_{1}}-\beta_{n}(1_{B}))A(1_{A_{1}}-\beta_{n}(1_{B}))$. Moreover,
    $\|d-\gamma_{n}(a')\|<\dfrac{\varepsilon}{2}$.

    Hence, $\beta_{n}(b_{1})+d$ is invertible in $A_{sa}$.

    Finally, we have

    $\|a-(\beta_{n}(b_{1})+d)\|$

    $\leq\|a-a'\|+\|a'-\gamma_{n}(a')-\beta_{n}\alpha(a')\|$

    $+\|\beta_{n}\alpha(a')-\beta_{n}(b_{1})\|+\|d-\gamma_{n}(a')\|$

    $<\varepsilon/2+\varepsilon/2+\varepsilon'<2\varepsilon$.

   Therefore, $A$ has real rank zero.
	\end{proof}

\begin{theorem}\label{thm:3.6}
Let $\Omega$ be a class of stably finite exact unital
${\rm C^*}$-algebras which have  $(n, m)$ comparison (with $m\neq 0$). Let $A$ be a unital separable simple ${\rm C^*}$-algebra. If  $A$ is asymptotically tracially in $\Omega$, then $A$  has   $(n, m)$ comparison (with $m\neq 0$).
\end{theorem}

\begin{proof}
  Let $a, b\in {\rm M}_{\infty}(A)_+$. We need to show that $\langle(a-\varepsilon)_+\rangle\leq\langle b\rangle$ for any $\varepsilon>0,$  and integers
 $m, n\geq 1$ such that $n \langle a\rangle +m\langle 1_A\rangle\leq   n\langle b\rangle.$

By Lemma \ref{lem:2.8},  we may assume that
$a, b\in A_+$ and $\|a\|\leq 1,\|b\|\leq 1$.

    We divide the proof into three cases.

   $(\textbf{I})$, we suppose that  $b$ is not Cuntz equivalent to a projection.

 Given $\delta>0$, and $n,~m\geq 1$ as above, since  $n \langle a\rangle +m\langle 1_A\rangle\leq   n\langle b\rangle$, hence, there exists
  $v=(v_{i,j})\in {\rm M}_{n+m}(A),  1\leq i\leq n+m, 1\leq j\leq n+m$ such that
 $$ \|v({\rm diag}(b\otimes 1_{n},~0\otimes 1_{m})){v}^*-{\rm diag}(a\otimes 1_{n}, 1_A\otimes 1_{m})\|<\delta.$$
We may assume that $\|v\|\leq M(\delta)$.

By Theorem \ref{thm:2.2} (3),  there is a non-zero positive element $d$ orthogonal to $b$  such that  $$(b-\delta/2)_++d\precsim b.$$

 With  $F=\{a, b, v_{i,j}: 1\leq i\leq n+m, 1\leq j\leq n+m\},$  and sufficiently small $\varepsilon'>0$, since $A$ is asymptotically tracially in $\Omega$, by Theorem \ref{thm:2.7}, there exist
a ${\rm C^*}$-algebra $B$  in $\Omega$ and completely positive  contractive linear maps  $\alpha:A\to B$ and  $\beta_k: B\to A$, and $\gamma_k:A\to A\cap\beta_n(B)^{\perp}$ such that

$(1)$ the map $\alpha$ is unital  completely positive   linear map, $\beta_k(1_B)$ and $\gamma_k(1_A)$ are all projections, and  $\beta_k(1_B)+\gamma_k(1_A)=1_A$, for all $k\in \mathbb{N}$,

$(2)$ $\|x-\gamma_k(x)-\beta_k(\alpha(x))\|<\varepsilon'$ for all $x\in F$, and for all $k\in {\mathbb{N}}$,

$(3)$ $\alpha$ is an $F$-$\varepsilon'$-approximate embedding,

$(4)$ $\lim_{k\to \infty}\|\beta_k(xy)-\beta_k(x)\beta_k(y)\|=0$ and $\lim_{k\to \infty}\|\beta_k(x)\|=\|x\|$ for all $x, y\in B$, and

$(5)$ $\gamma_k(1_A)\precsim d$ for all $k\in \mathbb{N}$.

Since $ \|v({\rm diag}(b\otimes 1_{n},~0\otimes 1_{m})){v}^*-{\rm diag}(a\otimes 1_{n}, 1_A\otimes 1_{m})\|<\delta,$
we have
 $$ \|\alpha\otimes id_{M_{n+m}}(v({\rm diag}(b\otimes 1_{n},~0\otimes 1_{m})){v}^*)-{\rm diag}(\alpha(a)\otimes 1_{n}, \alpha(1_A)\otimes 1_{m})\|<\delta.$$
  By $(3)$, we have
 $$ \|\alpha\otimes id_{M_{n+m}}(v)({\rm diag}(\alpha(b)\otimes 1_{n},~0\otimes 1_{m}))\alpha\otimes id_{M_{n+m}}({v}^*)-{\rm diag}(\alpha(a)\otimes 1_{n}, \alpha(1_A)\otimes 1_{m})\|$$$$<M(\delta)(n+m)^3\varepsilon'+(n+m)^2\delta<(n+m+1)^2\delta.$$
 By $(1)$ and  by Theorem \ref{thm:2.2} (1), we have
 $$n\langle(\alpha(a)-(n+m+1)^2\delta)_+\rangle+m\langle 1_B\rangle\leq m\langle \alpha(b)\rangle.$$
 Since $B\in \Omega$, we have $$\langle(\alpha(a)-(n+m+1)^2\delta)_+\rangle\leq \langle \alpha(b)\rangle.$$
   Since $\langle(\alpha(a)-(n+m+1)^2\delta)_+\rangle\leq \langle \alpha(b)\rangle$, for sufficiently small $\bar{\delta}$,  there exists  $w\in B$ such that $$\|w\alpha(b)w^*- (\alpha(a)-(n+m+1)^2\delta)_+\|<\bar{\delta}.$$
  We may assume that $\|w\|\leq H(\bar{\delta})$.

   Since $\|w\alpha(b)w^*- (\alpha(a)-(n+m+1)^2\delta)_+\|<\bar{\delta},$  we have
   $$\|\beta_k(w\alpha(b)w^*)- \beta_k((\alpha(a)-(n+m+1)^2\delta)_+)\|<\bar{\delta}.$$
   By $(4)$ there exist sufficiently large $N$, such that  for any $k>N$, we have $$\|\beta_k(w)\beta_k(\alpha(b))\beta_k(w^*)- \beta_k((\alpha(a)-(n+m+1)^2\delta)_+)\|$$$$<H(\bar{\delta})(n+m)^3\bar{\delta}+(n+m)^2\bar{\delta}<\delta.$$
 By Theorem  2.2 (1), we have $$\langle(\beta_k\alpha(a)-(n+m+2)^2\delta)_+\rangle\leq \langle (\beta_k\alpha(b)-2\delta)_+\rangle.$$
 Therefore, we have
\begin{eqnarray}
\label{Eq:eq1}
&&\langle(a-\varepsilon)_+\rangle \nonumber\\
&&\leq\langle\gamma_k(a)\rangle+\langle (\beta_k\alpha(a)-(n+m+2)^2\delta)_+\rangle\nonumber\\
&&\leq\langle\gamma_k(1_A)\rangle+\langle (\beta_k\alpha(a)-(n+m+2)^2\delta)_+)\rangle\nonumber\\
&&\leq\langle d\rangle+\langle (\beta_k(\alpha(b))-2\delta)_+\rangle\nonumber\\
&&\leq\langle (b-\delta/2)_+\rangle
+\langle d\rangle\leq\langle b\rangle.\nonumber
\end{eqnarray}

$(\textbf{II})$, we suppose that  $b$ is Cuntz equivalent to a projection and   $a$ is not Cuntz equivalent to a projection.
Choose a projection $p$ such that $b$ is  Cuntz equivalent to $p.$
We may assume that $b=p.$

By Theorem 2.1 (3), let $\delta>0$ (with $\delta<\varepsilon$),  there exists a non-zero positive element $d$  orthogonal to $a$  such that
 $\langle(a-\delta)_++ d\rangle \leq\langle a\rangle.$
Since $n \langle a\rangle +m\langle 1_A\rangle\leq   n\langle p\rangle$, we have $n\langle(a-\delta)_+ +d\rangle +m\langle 1_A\rangle\leq n\langle p\rangle.$ Hence,  there exists  $v=(v_{i,j})\in {\rm M}_{n+m}(A), ~1\leq i\leq n+m,~ 1\leq j\leq n+m$, such that
   $$\|v{\rm diag}(p\otimes 1_{n}, ~0\otimes 1_{m}) v^*-{\rm diag}(((a-\delta)_++ d)\otimes1_{n}, 1_A\otimes 1_m) \|<\delta.$$
We may assume that $\|v\|\leq M(\delta)$.

  Since $A$ is asymptotically tracially in $\Omega$,  by Theorem \ref{thm:2.7}, for $F=\{a,  p, d, {v_{i,j}}: 1\leq i\leq n+m, 1\leq j\leq n+m\},$  and  $\varepsilon'>0$, there exist
a ${\rm C^*}$-algebra $B$  in $\Omega$ and completely positive  contractive linear maps  $\alpha:A\to B$ and  $\beta_k: B\to A$, and $\gamma_k:A\to A\cap\beta_n(B)^{\perp}$ such that

$(1)$ the map $\alpha$ is unital  completely positive  linear map, $\beta_k(1_B)$ and $\gamma_k(1_A)$ are all projections, $\beta_k(1_B)+\gamma_k(1_A)=1_A$ for all $k\in \mathbb{N}$,

$(2)$ $\|x-\gamma_k(x)-\beta_k(\alpha(x))\|<\varepsilon'$ for all $x\in F$, and for all $k\in {\mathbb{N}}$.

$(3)$ $\alpha$ is an $F$-$\varepsilon'$-approximate embedding,

$(4)$ $\lim_{k\to \infty}\|\beta_k(xy)-\beta_k(x)\beta_k(y)\|=0$ and $\lim_{k\to \infty}\|\beta_k(x)\|=\|x\|$ for all $x, y\in B$.

Since $\|v{\rm diag} (p\otimes 1_{n}, 0\otimes 1_{m}) v^*-{\rm diag}(((a-\delta)_++ d)\otimes1_{n}, 1_A\otimes 1_m)\|<\delta$,
  by $(1)$,   we have
 $$ \|\alpha\otimes id_{M_{n+m}}(v({\rm diag}(p\otimes 1_{n}, 0\otimes 1_{m})){v}^*)-{\rm diag}(\alpha((a-\delta)_++ d)\otimes1_{n}, \alpha(1_A)\otimes 1_m)\|<\delta.$$
  By $(3)$, we have
 $$ \|\alpha\otimes id_{M_{n+m}}(v)({\rm diag}(\alpha(p)\otimes 1_{n}, 0\otimes 1_{m}))\alpha\otimes id_{M_{n+m}}({v}^*)$$$$-{\rm diag}(\alpha((a-\delta)_++ d)\otimes1_{n}, \alpha(1_A)\otimes 1_m)\|$$$$<M(\delta)(n+m)^3\varepsilon'+(n+m)^2\delta<(n+m+1)^2\delta.$$
 By $(1)$ and  by Theorem \ref{thm:2.2} (1), we have
 $$n\langle(\alpha((a-\delta)_++ d)-(n+m+1)^2\delta)_+\rangle+m\langle 1_B\rangle\leq m\langle \alpha(p)\rangle.$$
 Since $B\in \Omega$, we have $$\langle(\alpha((a-\delta)_++ d)-(n+m+1)^2\delta)_+\rangle\leq \langle \alpha(p_0)\rangle.$$
   Since $\langle(\alpha((a-\delta)_++ d)-(n+m+1)^2\delta)_+\rangle\leq \langle \alpha(p)\rangle$, for sufficiently small $\bar{\delta}$, there exist $w\in B$ such that $$\|w\alpha(b)w^*- (\alpha((a-\delta)_++ d)-(n+m+1)^2\delta)_+\|<\bar{\delta}.$$
   We may assume that $\|w\|\leq H(\bar{\delta})$.

   Since $\|w\alpha(p)w^*- (\alpha((a-\delta)_++ d)-(n+m+1)^2\delta)_+\|<\bar{\delta}$,  we have
   $$\|\beta_k(w\alpha(p)w^*)- \beta_k((\alpha((a-\delta)_++ d)-(n+m+1)^2\delta)_+)\|<\bar{\delta}.$$
   By $(4)$, there exists sufficiently large $n$, for $k>n$,  one has   $$\|\beta_k(w)\beta_k\alpha(p)\beta_k(w^*)- \beta_k((\alpha((a-\delta)_++ d)-(n+m+1)^2\delta)_+)\|$$$$<H(\bar{\delta})(n+m)^3\bar{\delta}+(n+m)^2\bar{\delta}<\delta.$$
By Theorem \ref{thm:2.2} (1), we have $$\langle(\beta_k\alpha((a-\delta)_++ d)-(n+m+2)^2\delta)_+\rangle\leq \langle \beta_k\alpha(p)\rangle.$$

Since $(a-\delta)_+$ orthogonal to $d$, by $(3)$ and
$(4)$, we may assume that $\beta_k\alpha((a-\delta)_+)$ orthogonal to $\beta_k\alpha(d)$.

 For sufficiently large  integer $k$, with  $G=\{\gamma_k(a), \gamma_k(p),  \gamma_k(v_{i,j}): 1\leq i\leq n+m, 1\leq j\leq n+m\},$  and any  sufficiently small $\varepsilon''>0$, let $E=\gamma_k(1_A)A\gamma_k(1_A)$, by Lemma \ref{lem:2.8}, $E$ is  asymptotically tracially in $\Omega$, there exist
a ${\rm C^*}$-algebra $D$  in $\Omega$ and completely positive  contractive linear maps  $\alpha':E\to D$ and  $\beta_k': D\to E$, and $\gamma_k':E\to E\cap\beta_k'(D)^{\perp}$ such that

$(1)'$ the map $\alpha'$ is unital  completely positive   linear map, $\beta_k'(1_D)$ and $\gamma_k'(1_E)$ are all projections, $\beta_k'(1_D)+\gamma_k'(1_E)=1_E$ for all $k\in \mathbb{N}$,

$(2)'$ $\|x-\gamma_k'(x)-\beta_k'(\alpha'(x))\|<\varepsilon''$ for all $x\in G$, and for all $k\in {\mathbb{N}}$,

$(3)'$ $\alpha'$ is an $G$-$\varepsilon''$-approximate embedding,

$(4)'$ $\lim_{k\to \infty}\|\beta_k'(xy)-\beta_k'(x)\beta_k'(y)\|=0$ and $\lim_{k\to \infty}\|\beta_k'(x)\|=\|x\|$ for all $x, y\in D$, and

$(5)'$ $\gamma_k'(\gamma(1_A))\precsim \beta_k\alpha(d)$ for all $k\in \mathbb{N}$.

Since $\| v{\rm diag}(p\otimes 1_{n},~0\otimes 1_{m}){v}^*-{\rm diag}(((a-\delta)_++d)\otimes 1_{n}, 1_A\otimes 1_m) \|<\delta$,
by $(2)$, we have
 \begin{eqnarray}
\label{Eq:eq1}
&&\|(\gamma_k\otimes id_{M_{n+m}}(v)+\beta_k\otimes id_{M_{n+m}}(\alpha\otimes id_{M_{n+m}}(v)))({\rm diag}(\gamma_k(p)\otimes 1_n, 0\otimes 1_{m}) \nonumber\\
&&+{\rm diag}(\beta_k\alpha(p)\otimes 1_n, 0\otimes 1_{m}))(\gamma_k\otimes id_{M_{n+m}}({v}^*)+\beta_k\otimes id_{M_{n+m}}(\alpha\otimes id_{M_{n+m}}({v}^*)))\nonumber\\
&&-{\rm diag}(\gamma_k((a-\delta)_++d)\otimes 1_n,\gamma_k(1_A)\otimes 1_m)\nonumber\\
 &&-{\rm diag}(\beta_k\alpha((a-\delta)_++d)\otimes 1_k, \beta_k\alpha(1_A)\otimes 1_m)\|\nonumber\\
&&<(n+m)^2\delta+M(\delta)(n+m)^3\varepsilon'<(n+m+1)^2\delta.\nonumber
\end{eqnarray}

 Since $\beta_k(B)\perp \gamma_k(A)$, therefore, we have
 \begin{eqnarray}
\label{Eq:eq1}
&&\|(\gamma_k\otimes id_{M_{n+m}}(v)){\rm diag}(\gamma_k(p)\otimes 1_n, 0\otimes 1_{m})(\gamma_k\otimes id_{M_{n+m}}({v}^*))\nonumber\\
&&-{\rm diag}(\gamma_k((a-\delta)_++d)\otimes 1_n,\gamma_k(1_A\otimes 1_m)\|
<(n+m+1)^2\delta.\nonumber
\end{eqnarray}

Since $ \|(\gamma_k\otimes id_{M_{k}}(v))({\rm diag}(\gamma_n(p)\otimes 1_n, ~0\otimes 1_{m}))(\gamma_k\otimes id_{M_{k}}({v}^*))
-{\rm diag}(\gamma_k((a-\delta)_++d)\otimes 1_n,\gamma_n(1_A\otimes 1_m)\|
<(n+m+1)^2\delta$,
  by $(1)'$ and $(3)'$,   we have
 \begin{eqnarray}
\label{Eq:eq1}
&&\|\alpha'\otimes id_{M_{n+m}}\gamma_k\otimes id_{M_{n+m}}(v){\rm diag}(\alpha'\gamma_k(p)\otimes 1_n, ~0\otimes 1_{m}))\nonumber\\
&&\alpha'\otimes id_{M_{n+m}}\gamma_k\otimes id_{M_{n+m}}(v^*)-{\rm diag}(\alpha'\gamma_k((a-\delta)_++d)\otimes 1_n, \alpha'\gamma_n(1_A)\otimes1_m)\|\nonumber\\
&&<M(\delta)(n+m)^3\varepsilon''+(n+m+1)^2(n+m)^2\delta<(n+m+1)^4\delta.\nonumber
\end{eqnarray}
 By $(1)$ and  Theorem \ref{thm:2.2} (1),  we have
  $$n\langle(\alpha'\gamma_k((a-\delta)_++d)-(n+m+2)^4\delta)_+\rangle+m\langle 1_D\rangle\leq n\langle (\alpha'\gamma_k(p)-\delta)_+\rangle. $$
Since $D\in \Omega,$  this implies  $$\langle (\alpha'\gamma_k((a-\delta)_++d)-(n+m+2)^4\delta)_+\rangle\leq \langle (\alpha'\gamma_k(p)-\delta)_+\rangle.$$
Since $\langle (\alpha'\gamma_k((a-\delta)_++d)-(n+m+2)^4\delta)_+\rangle\leq \langle (\alpha'\gamma_k(p)-\delta)_+\rangle,$  for sufficiently small $\bar{\bar{\delta}}$,
 there exist $w\in D$ such that
$$\|w(\alpha'\gamma_k(p)-\delta)_+w^*-(\alpha'\gamma_k((a-\delta)_++d)-(n+m+2)^4\delta)_+\|
<\bar{\bar{\delta}}.$$
We may assume that $\|w\|\leq H(\bar{\bar{\delta}})$.

By $(4)'$, there exists sufficiently large $k$, such that
$$\|\beta_k'(w)(\beta_k'\alpha'\gamma_k(p)-\delta)_+\beta_k(w^*)-
(\beta_k\alpha'\gamma_k((a-\delta)_++d)-(n+m+2)^4\delta)_+\|$$$$
<H(\bar{\bar{\delta}})(n+m)^3\bar{\bar{\delta}}+(n+m)^2\bar{\bar{\delta}}<\delta.$$
 By Theorem 2.2 (1),  we have
$$\langle (\beta_k'\alpha'\gamma_k((a-\delta)_++d)-(n+m+3)^4\delta)_+\rangle\leq \langle (\beta_k'\alpha'\gamma_k(p)-\delta)_+\rangle.$$

Therefore, we have
\begin{eqnarray}
\label{Eq:eq1}
&&\langle(a-\varepsilon)_+\rangle \nonumber\\
&&\leq\langle(\gamma_k(a)-(n+m+3)^6\delta)_+\rangle+\langle (\beta_k\alpha((a-\delta)_+)-(n+m+3)^2\delta)_+\rangle\nonumber\\
 &&\leq\langle\gamma_k'\gamma_k(a)\rangle+\langle (\beta_k'\alpha'\gamma_k((a-\delta)_+)-(n+m+3)^4\delta)_+\rangle\nonumber\\
  &&+\langle (\beta_k\alpha((a-\delta)_+)-(n+m+3)^2\delta)_+\rangle\leq\langle\gamma_k'\gamma_k(1_A)\rangle \nonumber\\
&&+\langle (\beta_k\alpha((a-\delta)_+)-(n+m+3)^2\delta)_+\rangle+\langle (\beta_k'\alpha'\gamma_k((a-\delta)_+)-(n+m+3)^4\delta)_+\rangle\nonumber\\
&&\leq\langle\beta_k\alpha(d)\rangle+
  \langle (\beta_k\alpha((a-\delta)_+)-(n+m+3)^4\delta)_+\rangle\nonumber\\
&&+\langle(\beta_k'\alpha'\gamma_k((a-\delta)_+)-(n+m+3)^4\delta)_+\rangle\nonumber\\
&&\leq\langle(\beta_k'\alpha'\gamma_k(p)-\delta)_+\rangle
+\langle \beta_k\alpha(p)\rangle\leq\langle p\rangle.\nonumber
\end{eqnarray}

$(\textbf{III})$, we suppose that both $a$ and $b$ are  Cuntz equivalent to projections.

   Choose  projections $p, q$ such that $a$ is Cuntz equivalent to $p$ and  $b$  is Cuntz equivalent to $q.$
We may assume that $a=p, b=q.$ Since
 $n\langle p\rangle+m\langle 1_A\rangle\leq n\langle q\rangle$. Hence,
    there exists
  $v=(v_{i,j})\in {\rm M}_{n+m}(A), 1\leq i\leq n+m, 1\leq j\leq n+m$ such that
 $$\| v{\rm diag}(q\otimes 1_{n}, 0\otimes 1_{m}){v}^*-{\rm diag}(p\otimes 1_{n}, 1_A\otimes1_m)\|<\delta.$$
 We may assume that $\|v\|\leq M(\delta)$.

Since $A$ is asymptotically tracially in $\Omega$, by Theorem \ref{thm:2.7},  for $F=\{p, q, v_{i,j}: 1\leq i\leq n+m, 1\leq j\leq n+m\},$ for  any $\varepsilon'>0,$
  there exist
a ${\rm C^*}$-algebra $B$  in $\Omega$ and completely positive  contractive linear maps  $\alpha:A\to B$ and  $\beta_k: B\to A$, and $\gamma_k:A\to A\cap\beta_n(B)^{\perp}$ such that

$(1)$ the map $\alpha$ is unital  completely positive linear map, $\beta_k(1_B)$ and $\gamma_k(1_A)$ are all projections $\beta_k(1_B)+\gamma_k(1_A)=1_A$ for all $k\in \mathbb{N}$,

$(2)$ $\|x-\gamma_k(x)-\beta_k(\alpha(x))\|<\varepsilon'$ for all $x\in F$, and for all $k\in {\mathbb{N}}$.

$(3)$ $\alpha$ is an $F$-$\varepsilon'$-approximate embedding,

$(4)$ $\lim_{k\to \infty}\|\beta_k(xy)-\beta_k(x)\beta_k(y)\|=0$ and $\lim_{k\to \infty}\|\beta_k(x)\|=\|x\|$ for all $x, y\in B$.

Since $\| v{\rm diag}(q\otimes 1_{n},~0\otimes 1_{m}){v}^*-{\rm diag}(p\otimes 1_{n},1_A\otimes 1_m)\|<\delta$,
by $(2)$, we have
 \begin{eqnarray}
\label{Eq:eq1}
&&\|(\gamma_k\otimes id_{M_{n+m}}(v)+\beta_k\otimes id_{M_{n+m}}(\alpha\otimes id_{M_{n+m}}(v)))({\rm diag}(\gamma_k(q)\otimes 1_n, ~0\otimes 1_{m}) \nonumber\\
&&+{\rm diag}(\beta_k\alpha(q)\otimes 1_n, ~0\otimes 1_{m}))(\gamma_k\otimes id_{M_{n+m}}({v}^*)+\beta_k\otimes id_{M_{n+m}}(\alpha\otimes id_{M_{n+m}}({v}^*)))\nonumber\\
&&-{\rm diag}((\gamma_k(p)\otimes 1_n, +\gamma_k(p)\otimes 1_m)\nonumber\\
&&-{\rm diag}(\beta_k\alpha(p)\otimes 1_n),\beta_k\alpha(1_A)\otimes 1_m)\|\nonumber\\
&&<(n+m)^2\delta+M(\delta)(n+m)^3\varepsilon'<(n+m+1)^2\delta.\nonumber
\end{eqnarray}

 Since $\beta_k(B)\perp \gamma_k(A)$, therefore, we have
 \begin{eqnarray}
\label{Eq:eq1}
&&\|(\gamma_k\otimes id_{M_{n+m}}(v)){\rm diag}(\gamma_k(q)\otimes 1_n, ~0\otimes 1_{m})(\gamma_k\otimes id_{M_{n+m}}({v}^*))\nonumber\\
&&-{\rm diag}(\gamma_k(p)\otimes 1_n, \gamma_k(1_A)\otimes1_m)\|
<(n+m+1)^2\delta.\nonumber
\end{eqnarray}
Since $\| v{\rm diag}(q\otimes 1_{n},~0\otimes 1_{m}){v}^*-{\rm diag}(p\otimes 1_{n}, 1_A\otimes 1_m)\|<\delta$,
by $(1)$ and $(3)$, we have
$$\|\alpha\otimes id_{M_{n+m}}(v)){\rm diag}(\alpha(q)\otimes 1_n, ~0\otimes 1_{m})
\alpha\otimes id_{M_{n+m}}({v}^*)-{\rm diag}(\alpha(p)\otimes 1_n,\alpha(1_A)\otimes 1_m) \|
$$$$<(n+m)^2\delta+M(\delta)(n+m)^3\varepsilon'<(n+m+1)^2\delta.$$
Since $\|\alpha\otimes id_{M_{n+m}}(v)){\rm diag}(\alpha(q)\otimes 1_n, ~0\otimes 1_{m})
\alpha\otimes id_{M_{n+m}}({v}^*)-{\rm diag}(\alpha(p)\otimes 1_n,\alpha(1_A)\otimes 1_m)\|
$ $<(n+m+1)^2\delta,$
by $(1)$ and  Theorem \ref{thm:2.2} (1), we have

$$n\langle(\beta_k\alpha(p)-(n+m+2)^2\delta)_+\rangle+m\langle\beta_k(1_B)\rangle\leq n\langle\beta_k\alpha(q)\rangle,$$

$$n\langle(\alpha(p)-(n+m+1)^2\delta)_+\rangle+m\langle1_B\rangle\leq n\langle\alpha(q)\rangle.$$
Since $B\in \Omega$, we have
$$\langle(\alpha(p)-(n+m+1)^2\delta)_+\rangle\leq \langle\alpha(q)\rangle.$$
Since $\langle(\alpha(p)-(n+m+1)^2\delta)_+\rangle\leq \langle\alpha(q)\rangle,$  for sufficiently small $\bar{\delta}$,  there exists $w\in B$ such that $$\|w\alpha(q)w^*-(\alpha(p)-(n+m+1)^2\delta)_+\|<\bar{\delta}.$$
We may assume that $\|w\|\leq H(\bar{\delta})$.

By $(4)$, there exists sufficiently large $k$, such that
$$\|\beta_k(w)\beta_n\alpha(q)\beta_k(w^*)-\beta_k((\alpha(p)-(n+m+1)^2\delta)_+)\|
$$$$<H(\bar{\delta})(n+m)^3\bar{\delta}+(n+m)^2\bar{\delta}<\delta.$$
By Theorem 2.1 (1) in \cite{EFF}, we  have
$$\langle(\beta_k \alpha(p)-(n+m+2)^2\delta)_+\rangle\leq \langle \beta_k\alpha(q)\rangle.$$

$(\textbf{III.I})$, if $(\beta_k\alpha (p)-(n+m+2)^2\delta)_+$ and $\beta_k\alpha(q)$  are not Cuntz equivalent to a pure positive element,  and $(\beta_k\alpha (p)-(n+m+2)^2\delta)_+$ cuntz equivalent to $\beta_k\alpha(q)$, then there exist projections $p_1, q_1$ such that $(\beta_n\alpha (p)-(n+m+2)^2\delta)_+\sim p_1$ and $\beta_n\alpha(q)\sim q_1$,
 then we have $n\langle p_1\rangle+m\langle \beta_k\alpha(q)\rangle\leq n\langle q_1\rangle=n\langle p_1\rangle.$
Since $m\langle \beta_k\alpha(q)\rangle\neq 0$,  and this
contradicts  the stable finiteness of $A$ (since ${\rm C^*}$-algebras in $\Omega$ are stably finite (cf.~\ proposition 4.2 in \cite{FL}).
So by Theorem \ref{thm:2.2} (2), there exist a nonzero positive element $s\in A$ and orthogonal to $(\beta_k\alpha (p)-(k+2)^2\delta)_+$ such that $(\beta_k\alpha (p)-(n+m+2)^2\delta)_++s\precsim \beta_k\alpha(q)$.

For sufficiently large  integer $k$, with  $G=\{\gamma_k(p), \gamma_k(q),  \gamma_k(v_{i,j}): 1\leq i\leq n+m, 1\leq j\leq n+m\},$  and any $\varepsilon''>0$, $E=\gamma_k(1_A)A\gamma_k(1_A)$, since $E$ is  asymptotically tracially in $\Omega$, there exist
a ${\rm C^*}$-algebra $D$  in $\Omega$ and completely positive  contractive linear maps  $\alpha':E\to D$ and  $\beta_k': D\to E$, and $\gamma_k':E\to E\cap\beta_k'(D)^{\perp}$ such that

$(1)'$ the map $\alpha'$ is unital  completely positive   linear map, $\beta_k'(1_D)$ and $\gamma_k'(1_E)$ are all projections, $\beta_k'(1_D)+\gamma_k'(1_E)=1_E$ for all $k\in \mathbb{N}$,

$(2)'$ $\|x-\gamma_k'(x)-\beta_k'(\alpha'(x))\|<\varepsilon''$ for all $x\in G$ and for all $k\in {\mathbb{N}}$,

$(3)'$ $\alpha'$ is an $G$-$\varepsilon''$-approximate embedding,

$(4)'$ $\lim_{k\to \infty}\|\beta_k'(xy)-\beta_k'(x)\beta_k'(y)\|=0$ and $\lim_{k\to \infty}\|\beta_k'(x)\|=\|x\|$ for all $x,~y\in D$, and

$(5)'$ $\gamma_k'(\gamma_k(1_A))\precsim s$ for all $k\in \mathbb{N}$.

Since
$$\|(\gamma_k\otimes id_{M_{n+m}}(v)){\rm diag}(\gamma_k(q)\otimes 1_n, ~0\otimes 1_{m})(\gamma_k\otimes id_{M_{n+m}}({v}^*))$$$$
-{\rm diag}(\gamma_k(p)\otimes 1_n,\gamma_k(1_A)\otimes 1_n )\|
<(n+m+1)^2\delta.$$

By $(3)'$, we have
\begin{eqnarray}
\label{Eq:eq1}
&&\|\alpha'\otimes id_{M_{n+m}}\gamma_k\otimes id_{M_{n+m}}(v)){\rm diag}(\alpha'\gamma_k(q)\otimes 1_n, ~0\otimes 1_{m})\nonumber\\
&&\alpha'\otimes id_{M_{n+m}}\gamma_k\otimes id_{M_{n+m}}(v^*)-{\rm diag}(\alpha'\gamma_k(p)\otimes 1_n, \alpha'\gamma_k(1_A)\otimes 1_n)\|\nonumber\\
&&<M(\delta)(n+m)^3\varepsilon''+(n+m)^2(n+m+1)^2\delta<(n+m+1)^4\delta.\nonumber
\end{eqnarray}
By Theorem \ref{thm:2.2} (1), we  have
$$n\langle(\alpha'\gamma_k(p)-(n+m+2)^4\delta)_+\rangle+m\langle1_D\rangle\leq n\langle (\alpha'\gamma_k(q)-\delta)_+
\rangle. $$
Since $D\in \Omega$, we have
$$\langle(\alpha'\gamma_k(p)-(n+m+2)^4\delta)_+\rangle\leq \langle (\alpha'\gamma_k(q)-\delta)_+
\rangle. $$
For sufficiently small $\bar{\bar{\delta}}$, there exists $w\in D$ such that
$$\|w(\alpha'\gamma_k(q)-\delta)_+w^*-(\alpha'\gamma_k(p)-(n+m+2)^4\delta)_+\|
<\bar{\bar{\delta}}.$$
We may assume that $\|w\|\leq H(\bar{\bar{\delta}})$.

By $(4')$, there exist sufficiently large $n$, such that
$$\|\beta_k'(w)\beta_n'(\alpha'\gamma_k(q)-\delta)_+\beta_k(w^*)
-\beta_k((\alpha'\gamma_k(p)-(n+m+2)^4\delta)_+)\|
$$$$<H(\bar{\bar{\delta}})(n+m)^3\bar{\bar{\delta}}+(n+m)^2\bar{\bar{\delta}}<\delta.$$
We  have   $$\langle (\beta_k'\alpha'\gamma_k(p)-(n+m+3)^4\delta)_+\rangle\leq \langle (\beta_k'\alpha'\gamma_k(q)-\delta)_+\rangle.$$

Therefore, if $\varepsilon'$, are small enough, then
\begin{eqnarray}
\label{Eq:eq1}
&&\langle (p-\varepsilon)_+\rangle \nonumber\\
&&\leq\langle\gamma_k(p)-(n+m+3)^6\delta)_+\rangle+\langle (\beta_k\alpha(p)-(n+m+3)^2\delta)_+\rangle\nonumber\\
 &&\leq\langle(\gamma_k'\gamma_n(p)-4\delta)_+\rangle+\langle (\beta_k'\alpha'\gamma_k(p)-(n+m+3)^4\delta)_+\rangle\nonumber\\
  &&+\langle (\beta_k\alpha(p)-(n+m+3)^2\delta)_+\rangle\leq\langle\gamma_k'\gamma_n(1_A)\rangle+
  \langle (\beta_k\alpha(p)-(n+m+3)^2\delta)_+\rangle \nonumber\\
&&+\langle (\beta_k'\alpha'\gamma_k(p)-(n+m+3)^4\delta)_+\rangle\nonumber\\
&&\leq\langle s\rangle+
  \langle (\beta_k\alpha(p)-(n+m+2)^2\delta)_+\rangle+\langle (\beta_k'\alpha'\gamma_k(p)-\delta)_+\rangle\nonumber\\
&&\leq\beta_k\alpha(q)+(\beta_k'\alpha'\gamma_k(q)-\delta)_+ \leq \langle q\rangle.\nonumber
\end{eqnarray}

$(\textbf{III.II})$, We suppose that $(\beta_k\alpha (p)-(n+m+2)^2\delta)_+$ is a purely positive element, then, by Theorem \ref{thm:2.2} (3),   there is a non-zero positive element $d$ orthogonal to $(\beta_k\alpha (p)-(n+m+2)^2\delta)_+$  such that  $(\beta_k\alpha (p)-(n+m+3)^2\delta)_++d\precsim (\beta_k\alpha (p)-(n+m+2)^2\delta)_+$.

With the same argument as $(\textbf{III.I})$, for sufficiently large integer $k$, with  $G=\{\gamma_k(p), \gamma_k(q),  \gamma_k(v_{i,j}): 1\leq i\leq n+m, 1\leq j\leq n+m\},$  and  $\varepsilon''>0$, $E=\gamma_k(1_A)A\gamma_k(1_A)$, since $E$ is  asymptotically tracially in $\Omega$, there exist
a ${\rm C^*}$-algebra $D$  in $\Omega$ and completely positive  contractive linear maps  $\alpha':E\to D$ and  $\beta_k': D\to E$, and $\gamma_k':E\to E\cap\beta_k'(D)^{\perp}$ such that
$$\langle (\beta_k'\alpha'\gamma_k(p)-(n+m+3)^4\delta)_+\rangle\leq \langle (\beta_k'\alpha'\gamma_k(q)-\delta)_+\rangle.$$

Therefore, if $\varepsilon'$, are small enough, then
\begin{eqnarray}
\label{Eq:eq1}
&&\langle (p-\varepsilon)_+\rangle \nonumber\\
&&\leq\langle(\gamma_k(p)-(n+m+3)^6\delta)_+\rangle+\langle (\beta_k\alpha(p)-(n+m+3)^2\delta)_+\rangle\nonumber\\
 &&\leq\langle(\gamma_k'\gamma_k(p)-4\delta)_+\rangle+\langle (\beta_k'\alpha'\gamma_k(p)-(n+m+3)^4\delta)_+\rangle\nonumber\\
  &&+\langle (\beta_k\alpha(p)-(n+m+3)^2\delta)_+\rangle\leq\langle\gamma_k'\gamma_n(1_A)\rangle+
  \langle (\beta_k\alpha(p)-(n+m+3)^2\delta)_+\rangle \nonumber\\
&&+\langle (\beta_k'\alpha'\gamma_k(p)-(n+m+3)^4\delta)_+\rangle\nonumber\\
&&\leq\langle s\rangle+
  \langle (\beta_k\alpha(p)-2\delta)_+\rangle+\langle (\beta_k'\alpha'\gamma_k(p)-2\delta)_+\rangle\nonumber\\
&&\leq\langle \beta_k\alpha(q)\rangle+\langle (\beta_k'\alpha'\gamma_k(q)-\delta)_+\rangle \leq \langle q\rangle.\nonumber
\end{eqnarray}

$(\textbf{III.III})$, we suppose that   $(\beta_k\alpha (p)-(n+m+2)^2\delta)_+$ is Cuntz equivalent to a projection and   $\beta_k\alpha(q)$ is not Cuntz equivalent to a projection.

By  Theorem \ref{thm:2.2} (3),  there exists a non-zero positive element $d$  orthogonal to $\beta_k\alpha(q)$ such that
 $\langle(\beta_k\alpha (q)-\delta)_++ d\rangle \leq\langle\beta_k\alpha(q)\rangle.$

 With the same argument as $(\textbf{III.I})$, for sufficiently large integer $k$, with  $G=\{\gamma_k(p),~ \gamma_k(q),~ \gamma_k(v_{i,j}): ~1\leq i\leq n+m,~ 1\leq j\leq n+m\},$  and any sufficiently small $\varepsilon''>0$,  $E=\gamma_k(1_A)A\gamma_k(1_A)$, since $E$ is  asymptotically tracially in $\Omega$, there exist
a ${\rm C^*}$-algebra $D$  in $\Omega$ and completely positive  contractive linear maps  $\alpha':E\to D$ and  $\beta_k': D\to E$, and $\gamma_k':E\to E\cap\beta_k'(D)^{\perp}$,   such that
$$\langle (\beta_k'\alpha'\gamma_k(p)-(n+m+3)^4\delta)_+\rangle\leq \langle (\beta_k'\alpha'\gamma_k(q)-\delta)_+\rangle.$$

Therefore, we have
\begin{eqnarray}
\label{Eq:eq1}
&&\langle (p-\varepsilon)_+\rangle \nonumber\\
&&\leq\langle\gamma_k'(p)-(n+m+3)^6\delta)_+\rangle+\langle (\beta_k\alpha(p)-(n+m+3)^2\delta)_+\rangle\nonumber\\
 &&\leq\langle(\gamma_n'\gamma_k(p)-4\delta)_+\rangle+\langle (\beta_k'\alpha'\gamma_k(p)-(n+m+3)^4\delta)_+\rangle\nonumber\\
  &&+\langle (\beta_k\alpha(p)-\delta)_+\rangle\leq\langle\gamma_k'\gamma_n(1_A)\rangle+
  \langle (\beta_k\alpha(p)-\delta)_+\rangle \nonumber\\
&&+\langle (\beta_k'\alpha'\gamma_k(p)-(n+m+3)^4\delta)_+\rangle\nonumber\\
&&\leq\langle s\rangle+
  \langle (\beta_k\alpha(p)-\delta)_+\rangle+\langle (\beta_k'\alpha'\gamma_k(p)-(n+m+3)^4\delta)_+\rangle\nonumber\\
&&\leq \langle \beta_k\alpha(q)\rangle+\langle (\beta_k'\alpha'\gamma_k(q)-\delta)_+\rangle\leq \langle q\rangle.\nonumber
\end{eqnarray}
\end{proof}

\begin{corollary}\label{cor:3.4}
Let $\Omega$ be a class of stably finite exact  unital
${\rm C^*}$-algebras which have the radius of comparison $n$ (with $n\neq 0$). Let $A$ be a unital separable simple ${\rm C^*}$-algebra. If  $A$ is asymptotically tracially in $\Omega$, then ${\rm rc}(A)\leq n$.

\end{corollary}
\begin{proof}By Theorem \ref{thm:3.6}.

\end{proof}

 \end{document}